\documentclass[12pt]{amsart}

\usepackage{psfrag}
\usepackage{color}
\usepackage{tikz}
\usetikzlibrary{matrix,arrows}
\usepackage{graphicx,graphics}
\usepackage{fullpage,amssymb,amsfonts,amsmath,amstext,amsthm,amscd,verbatim,enumerate}
\usepackage[T1]{fontenc}

\begin{document}

\newtheorem{theorem}{Theorem}[section]
\newtheorem{result}[theorem]{Result}
\newtheorem{fact}[theorem]{Fact}
\newtheorem{example}[theorem]{Example}
\newtheorem{conjecture}[theorem]{Conjecture}
\newtheorem{lemma}[theorem]{Lemma}
\newtheorem{proposition}[theorem]{Proposition}
\newtheorem{corollary}[theorem]{Corollary}
\newtheorem{facts}[theorem]{Facts}
\newtheorem{props}[theorem]{Properties}
\newtheorem*{thmA}{Theorem A}
\newtheorem{ex}[theorem]{Example}
\theoremstyle{definition}
\newtheorem{definition}[theorem]{Definition}
\newtheorem{remark}[theorem]{Remark}
\newtheorem*{defna}{Definition}

\newcommand{\notes} {\noindent \textbf{Notes.  }}
\newcommand{\note} {\noindent \textbf{Note.  }}
\newcommand{\defn} {\noindent \textbf{Definition.  }}
\newcommand{\defns} {\noindent \textbf{Definitions.  }}
\newcommand{\x}{{\bf x}}
\newcommand{\z}{{\bf z}}
\newcommand{\B}{{\bf b}}
\newcommand{\V}{{\bf v}}
\newcommand{\T}{\mathbb{T}}
\newcommand{\Z}{\mathbb{Z}}
\newcommand{\Hp}{\mathbb{H}}
\newcommand{\D}{\mathbb{D}}
\newcommand{\R}{\mathbb{R}}
\newcommand{\N}{\mathbb{N}}
\renewcommand{\B}{\mathbb{B}}
\newcommand{\C}{\mathbb{C}}
\newcommand{\ft}{\widetilde{f}}
\newcommand{\dt}{{\mathrm{det }\;}}
 \newcommand{\adj}{{\mathrm{adj}\;}}
 \newcommand{\0}{{\bf O}}
 \newcommand{\av}{\arrowvert}
 \newcommand{\zbar}{\overline{z}}
 \newcommand{\xbar}{\overline{X}}
 \newcommand{\htt}{\widetilde{h}}
\newcommand{\ty}{\mathcal{T}}
\renewcommand\Re{\operatorname{Re}}
\renewcommand\Im{\operatorname{Im}}
\newcommand{\tr}{\operatorname{Tr}}

\newcommand{\ds}{\displaystyle}
\numberwithin{equation}{section}

\renewcommand{\theenumi}{(\roman{enumi})}
\renewcommand{\labelenumi}{\theenumi}

\title{Unicritical Blaschke Products and Domains of Ellipticity}

\author{Alastair Fletcher}
\address{Department of Mathematical Sciences, Northern Illinois University, DeKalb, IL 60115-2888. USA}
\email{fletcher@math.niu.edu}

\maketitle

\begin{abstract}
Elliptic M\"obius transformations of the unit disk are those for which there is a fixed point in $\D$. It is not hard to classify which M\"obius transformations are elliptic in terms of the parameters. The set of parameters can be identified with the solid torus $S^1 \times \D$, and the set of elliptic parameters is called the domain of ellipticity. In this paper, we study the domain of ellipticity for non-trivial unicritical Blaschke products. 
We will also study the set corresponding to the Mandelbrot set for this family, and show how it can be obtained from the domain of ellipticity by adding one point.
\end{abstract}

\section{Introduction}

\subsection{Families of holomorphic mappings}

Let $\mathcal{F}$ be a family of holomorphic functions defined on a domain $U\subset \C$ and parameterized by some set $Y$, that is, for each $y\in Y$, there is a function $f_y \in \mathcal{F}$. We often want to be able to classify some aspect of the dynamics of functions in $\mathcal{F}$ in terms of the parameter space $Y$. 

For example, the study of the family of all quadratic polynomials $Az^2+Bz+C$ can be reduced to the study of those of the family $\mathcal{F}  = \{ f_c(z):= z^2+c :c\in \C \}$ by conjugating by a suitable linear map. Conjugation does not change the dynamics in a significant way. The parameter space is a copy of the complex plane $Y = \{ c\in \C : f_c \in \mathcal{F} \}$. In the family of quadratic functions, the Julia set is either connected or totally disconnected. The connectedness locus $\mathcal{M}$ is defined to be $\{ c\in \C : J(f_c) \text{ is connected}\}$ and is called the Mandelbrot set.

Very often, we restrict a given family to a certain sub-family to help give insight into the behaviour of the whole family. For example, among all degree $d$ polynomials, the unicritical polynomials of the form $z^d+c$ are of particular interest. The  connectedness locus in parameter space is the so-called Multibrot set, see for example \cite{EMS}. Analogously, in the antiholomorphic setting, functions of the form $\overline{z}^d+c$ are studied, giving rise to the Multicorn set in parameter space, see for example \cite{M1}.

In this paper, we will focus on the hyperbolic versions of unicritical polynomials, namely unicritical Blaschke products of the form 
\[ B(z) = \left ( \frac{z-w}{1-\overline{w}z} \right )^n,\]
where $w\in \D$ and $n\geq 2$. We will first recall the classification theory of M\"obius mappings.

\subsection{M\"obius transformations}

Every M\"obius transformation of the unit disk $\D$ can be written in the form
\[ A(z) = e^{i\theta} \left ( \frac{z-w}{1-\overline{w}z} \right) \]
for some $\theta \in[0,2\pi)$ and $w\in \D$. This can be written alternatively as
\[ A(z) = \frac{Ce^{i\theta/2}z - Ce^{i\theta/2}w}{ Ce^{-i\theta/2} - Ce^{-i\theta/2}\overline{w}z},\]
where $C=(1-|w|^2)^{-1/2}$. The significance of writing $A$ in this form is that the matrix representing $A$, given by
\[ \begin{pmatrix}  Ce^{i\theta/2} & -Ce^{i\theta/2}w \\ - Ce^{-i\theta/2}\overline{w} &  Ce^{-i\theta/2} \end{pmatrix},\]
has determinant $1$ and trace-squared equal to
\begin{equation} \label{eq:tsq} \tau(A) =  \left ( \frac{e^{i\theta/2}}{\sqrt{1-|w|^2}} + \frac{e^{-i\theta/2}}{\sqrt{1-|w|^2}} \right )^2 = \frac{2(1+\cos\theta)}{1-|w|^2}.\end{equation}
We recall that M\"obius transformations of $\D$ can be classified as follows:
\begin{enumerate}[(i)]
\item $A$ is called {\it hyperbolic} if $A$ has two fixed points on $\partial \D$ and none in $\D$,
\item $A$ is called {\it parabolic} if $A$ has one fixed point on $\partial \D$ and none in $\D$,
\item $A$ is called {\it elliptic} if $A$ has no fixed points on $\partial \D$ and one in $\D$.
\end{enumerate}
Note that by the Schwarz-Pick Lemma (see for example \cite{BM}), if $A$ is not the identity, $A$ can have a maximum of one fixed point in $\D$, and so these three cases provide a complete classification of M\"obius transformations of $\D$. This classification can also be expressed in terms of $\tau$:
\begin{enumerate}[(i)]
\item $A$ is hyperbolic if and only if $\tau(A) >4$,
\item $A$ is parabolic if and only if $\tau(A) =4$,
\item $A$ is elliptic if and only if $0 \leq \tau(A)<4$.
\end{enumerate}
We see from \eqref{eq:tsq} that $\tau(A)$ is real when $A$ is represented in the normalized form as given above.
Hence if we fix $\theta\in [0,2\pi )$, the set of $w$-values for which $A$ is elliptic is given by the disk 
\[ \left \{w\in \D : |w| <\sin(\theta/2) \right \}.\]
Note this set is empty when $\theta =0$. Since the set of parameters for M\"obius transformations of $\D$ can be parameterized by the solid torus $ S^1 \times \D$, the {\it domain of ellipticity} is given by the open set
\begin{equation} \label{eq:e} E:= \left \{ (e^{i\theta},w) \in S^1 \times \D : |w| < \sin(\theta/2) \right \}.\end{equation}
The boundary of $E$ gives the set of parabolic parameters by the classification in terms of $\tau$.

\subsection{Blaschke products}

A {\it finite Blaschke product} is a function $B:\D \to \D$ of the form
\begin{equation}\label{eq:bl} B(z) = e^{i\theta} \prod_{i=1}^n \left ( \frac{z-w_i}{1-\overline{w_i}z }\right ),\end{equation}
for some $\theta \in [0,2\pi)$ and $w_i \in \D$ for $i=1,\ldots,n$. A Blaschke product can also be viewed as a rational function on $\C \cup \{ \infty \}$. We call a Blaschke product non-trivial if $n\geq 2$.

Every finite degree self-mapping of $\D$ is a finite Blaschke product \cite[p.19]{BM}, and so they can be viewed as analogues for polynomials in the disk. Again by the Schwarz-Pick Lemma, $B$ can have at most one fixed point in $\D$. If $z_0$ is a fixed point of $B$, then it is straightforward to show that $1/\overline{z_0}$ is also a fixed point of $B$. Hence all but possibly two (with the convention that infinity is a fixed point if some $w_i=0$) of the fixed points of $B$ must lie on $\partial \D$. 

The Denjoy-Wolff Theorem \cite[p.58]{Milnor} states that if $f:\D\to \D$ is holomorphic then there is some $z_0 \in \overline{\D}$ such that $f^n(z) \to z_0$ for every $z\in \D$. We call such a point a {\it Denjoy-Wolff point} of $f$. There is a classification of finite Blaschke products in analogy with that for M\"obius transformations:
\begin{enumerate}[(i)]
\item $B$ is called {\it hyperbolic} if the Denjoy-Wolff point $z_0$ of $B$ lies on $\partial \D$ and $|B'(z_0)|<1$,
\item $B$ is called {\it parabolic} if the Denjoy-Wolff point $z_0$ of $B$ lies on $\partial \D$ and $|B'(z_0)|=1$,
\item $B$ is called {\it elliptic} if the Denjoy-Wolff point $z_0$ of $B$ lies in $\D$. In this case, we must have $|B'(z_0)|<1$.
\end{enumerate}

\subsection{Dynamics of Blaschke products}

The Fatou set of a rational map $f:\overline{\C} \to \overline{\C}$ is the set
\[ F(f) = \{z\in \overline{\C} : \text{ the family } (f^n)_{n=1}^{\infty} \text{ is normal in some neighbourhood of } z \}.\]
This is the domain of stable behaviour of the iterates of $f$. The set of chaotic behaviour is the Julia set $J(f)$ and is given by $\overline{\C} \setminus F(f)$. The Fatou set is always open and hence the Julia set is always closed.

For finite Blaschke products, the Julia set is always contained in $\partial \D$ and is either the whole of $\partial \D$ or a Cantor subset of $\partial \D$. These two cases can be characterized as follows; see \cite[p.58]{CG} and \cite{BCH,CDP} as well as \cite{KP} for a discussion of this characterization.

\begin{theorem}
Let $B$ be a non-trivial finite Blaschke product. Then:
\begin{enumerate}[(i)]
\item if $B$ is elliptic, $J(B) = \partial \D$,
\item if $B$ is hyperbolic, $J(B)$ is a Cantor subset of $\D$,
\item if $B$ is parabolic and $z_0 \in \partial \D$ is the Denjoy-Wolff point of $B$, $J(B) = \partial \D$ if $B''(z_0) = 0$ and $J(B)$ is a Cantor subset of $\partial \D$ if $B''(z_0)\neq 0$.
\end{enumerate}
\end{theorem}

We remark that the parabolic case is the most subtle. Denoting by $d$ the hyperbolic distance in $\D$, a finite Blaschke product is said to be of {\it zero hyperbolic step} if there exists $z\in \D$ such that $\lim_{n\to \infty} d(B^n(z),B^{n+1}(z)) =0$. If this holds for some $z\in \D$, then it holds for every point in $\D$. If $B$ is not of zero hyperbolic step, then since $d(B^n(z),B^{n+1}(z))$ is non-increasing by the Schwarz-Pick Lemma, we have $\lim _{n\to \infty} d(B^n(z),B^{n+1}(z)) >0$. In this case, $B$ is said to be of {\it positive hyperbolic step}. See \cite{CDP}, as well as \cite{KP}, for more details.

\begin{theorem}[\cite{BCH,CDP}]\label{thm:f''}
A parabolic finite Blaschke product $B$ is of zero hyperbolic step if and only if $J(B) = \partial \D$ if and only if $B''(z_0)=0$, where $z_0$ is the Denjoy-Wolff point of $B$.
\end{theorem}

\subsection{Statement of results}

We will specialize to the case of unicritical Blaschke products, that is, Blaschke products for which there is a unique point $w\in \D$ such that $B'(w)=0$. If $B$ has degree $n$, then the critical point is taken with multiplicity $n-1$. Further, a unicritical Blaschke product is necessarily finite. We will want to conjugate by M\"obius mappings to reduce the unicritical Blaschke product to a simple form.

\begin{proposition}
\label{prop1}
Let $B:\D \to \D$ be a unicritical Blaschke product of degree $n$. Then $B$ is conjugate by a M\"obius mapping to a unique Blaschke product of the form
\[ \left ( \frac{z-w}{1-\overline{w}z} \right )^n,\]
for some $w\in \D$ where $\arg(w) \in [0,\frac{2\pi}{n-1})$.
\end{proposition}

We can therefore consider the set 
\[ \mathcal{B}_n  = \left \{ B_w(z):= \left ( \frac{z-w}{1-\overline{w}z} \right )^n : w\in \D, \arg(w) \in \left [0,\frac{2\pi}{n-1} \right ) \right \}\]
of normalized unicritical Blaschke products of degree $n$, which is parameterized by the sector 
\[ S_n = \{ w\in \D: \arg(w) \in [0,\frac{2\pi}{n-1}) \} .\]
We will be interested in two subsets of parameter space, the set of ellipticity
\[ \mathcal{E}_n = \{w\in S_n : J(B_w) \text{ is elliptic}\},\]
and the connectedness locus
\[ \mathcal{M}_n = \{ w\in S_n : J(B_w) =\partial \D \}.\]
We denote by $\widetilde{\mathcal{E}_n}$  the set obtained by unioning $\mathcal{E}_n$ with $R_j(\mathcal{E}_n)$ for $j=1,\ldots,n-2$, where $R_j$ is the rotation through angle $2\pi j/(n-1)$, and denote by $\widetilde{\mathcal{M}_n}$ the corresponding set for $\mathcal{M}_n$.

\begin{theorem}\label{thm:1}
Let $n\geq 2$.
The set $\widetilde{\mathcal{E}_n} \subset \D$ is a starlike domain about $0$ which contains the disk $ \{w\in \D : |w| <\frac{n-1}{n+1} \}$
\end{theorem}

We next discuss the connectedness locus $\mathcal{M}_n$.
Since every elliptic parameter gives rise to a connected Julia set, the question of which parabolic elements give a connected Julia set remains. It follows from the proof of Theorem \ref{thm:1} that the parabolic elements are parameterized by the relative boundary of $\mathcal{E}_n$ in $S_n$.

\begin{theorem}\label{cor:1}
The set $\mathcal{M}_n$ is the union of $\mathcal{E}_n$ together with one point on the relative boundary of $\mathcal{E}_n$ in $ S_n$, given by $\partial E_n \cap \{w:|w| = \frac{n-1}{n+1} \}$.
\end{theorem}

The set $\mathcal{M}_n$ plays the same role for the family of Blaschke products considered in this paper as the Mandelbrot set does for the family of quadratic polynomials of the form $z^2+c$. Now, $z^2+c_1$ and $z^2+c_2$  are not conjugate unless $c_1=c_2$. Analogous to this fact, given two distinct Blaschke products with Julia sets equal to $\partial \D$, when restricted to $\partial \D$ they are not absolutely continuously conjugate by a result of Shub and Sullivan \cite{SS} (see also \cite{Hamilton}) unless they are conjugate by a M\"obius mapping. On the other hand, by Proposition \ref{prop1} no two elements of $\mathcal{B}_n$ are conjugate by M\"obius mappings.

We remark that every degree two Blaschke product is unicritical, and leave open the question of describing the conjugacy class of an element of $\mathcal{B}_2$ in the space of all degree two Blaschke products.

An application of these results will be given in \cite{F}. There, it will be shown that in the neighbourhood of a fixed point of a quasiregular mapping in the plane with constant complex dilatation and of any local index, the behaviour of the iterates can be determined by a conjugate of a Blaschke product. Hence knowing when such a Blaschke product has Julia set equal to $\partial \D$ or a Cantor subset of $\partial \D$ has consequences for the dynamics of the quasiregular mapping.

The author would like to thank Zhuan Ye, Cao Chunlei and Katherine Plikuhn for interesting discussions and comments on an earlier draft of this paper. The diagrams used in this paper were produced by Katherine Plikuhn.

\section{Proof of Proposition \ref{prop1}}

Let $B$ be a unicritical Blaschke product of degree $n$ with critical point $z_0$. Let $A(z) = \frac{z-z_0}{1-\overline{z_0}z}$ and define $B_1 = A\circ B\circ A^{-1}$. It is not hard to see that $B_1$ is unicritical with critical point $0$. 

Now, $z^n$ is also a unicritical Blaschke product of degree $n$ with critical point $0$. Hence by \cite[Corollary 1]{Zakeri}, there exists a M\"obius mapping $M$ such that $B_1(z) = M(z^n)$. If we now define $B_2 = M^{-1}\circ B_1 \circ M$, we have $B_2(z) = [M(z)]^n$.

Writing $M(z) = e^{i\theta}\left ( \frac{z-u}{1-\overline{u}z} \right )$, let $R(z) = e^{i\alpha}z$ with $ \alpha $ to be determined. Then
\begin{align*}
RB_2R^{-1}(z) &= e^{i\alpha}\left ( e^{i\theta} \left ( \frac{e^{-i\alpha}z-u}{1-\overline{u}e^{-i\alpha}z} \right ) \right )^n \\
&=e^{i(n\theta +(1-n)\alpha)} \left ( \frac{ z-ue^{i\alpha}}{1-\overline{ue^{i\alpha}}z }\right)^n.
\end{align*}
We can therefore choose $\alpha$ so that $n\theta +(1-n)\alpha \in 2k\pi$ for $k\in \Z$ and also so that $\arg(ue^{i\alpha}) \in [0,2\pi/(n-1))$. This shows that every unicritical Blaschke product is conjugate by a M\"obius mapping to one in the set $\mathcal{B}_n$.

Finally, we need to show that two elements of $\mathcal{B}_n$ are not conjugate by a M\"obius map. To that end, let $w\in S_n$ and let $B_w$ be the corresponding unicritical Blaschke product. Let $A$ be a M\"obius mapping and suppose $B=AB_wA^{-1}\in \mathcal{B}_n$. We have $B'(z) =0$ if and only if $B_w'(A^{-1}(z))=0$ if and only if $z=A(w)$, and so $B$ is unicritical with critical point $A(w)$. However, on the one hand we have $B(A(w)) = A(0)$, but we must also have $B(A(w))=0$ because $B\in \mathcal{B}_n$. Hence $A(0)=0$ and so $A$ is a rotation. Suppose $A(z) = e^{i \beta}$. Then
\[ B(z) = AB_wA^{-1}(z) = e^{(1-n)i\beta} \left ( \frac{ z-we^{i\beta}}{1-\overline{we^{i\beta}}z} \right ),\]
and so to have $B\in \mathcal{B}_n$ we need $(1-n)\beta \in 2k\pi$ for $k\in \Z$. This means $e^{i\beta}$ must be an $(n-1)$'th root of unity, and the uniqueness claim follows since no two distinct elements of $S_n$ are related by such a rotation.

\section{Proof of Theorem \ref{thm:1}}

\subsection{Fixing notation and outline of the proof}

Let $n \in\N$ with $n\geq 2$ be fixed. We write $w\in \D$ in polar form $w=se^{i\psi}$ and will analyze what happens when $\psi$ is fixed and $s\in [0,1)$ is varied. 
We denote by $A$ the M\"obius transformation
\[ A(z) = \left ( \frac{z-w}{1-\overline{w}z} \right)\]
and by $B$ the Blaschke product
\[ B(z)=B_w(z) = [A(z)]^n.\]

If $z_0$ is a fixed point of a Blaschke product, then $1/\overline{z_0}$ is also a fixed point. It follows that
a Blaschke product of degree $n$ either has $n-1$ repelling fixed points on $\partial \D$ and a pair of attracting fixed points in $\D$ and $\overline{\C} \setminus \overline{\D}$, or $n$ repelling fixed points on $\partial \D$ and one attracting fixed point on $\partial \D$, or all fixed points on $\partial \D$ with one of them neutral. Therefore, we will be interested in studying the set of points on $\partial \D$ where $|B'(z)|\leq1$ and when this set contains a fixed point. This will then determine whether the corresponding Blaschke product is elliptic, hyperbolic or parabolic.

\subsection{The subset of $\partial \D$ where $|B'(z)|<1$}

The derivative of $B$ is
\begin{equation}\label{eq:f'} B'(z) = nA(z)^{n-1}A'(z).\end{equation}
Since for $|z|=1$ we have $|A(z)|=1$, it follows that if $z=e^{i\phi}$,
\begin{equation} \label{eq:deriv} |B'(z)| = \frac{n(1-|w|^2)}{|1-\overline{w}z|^2} = \frac{n(1-s^2)}{ |1-se^{i(\phi-\psi)}|^2} = \frac{ n(1-s^2)}{1+s^2-2s\cos(\phi - \psi)}.\end{equation}

\begin{definition}\label{def:j}
Denote by $K=K(s)$ the set of points on $\partial \D$ where $|B'(z)|\leq 1$. More precisely,
\[ K = \left \{ e^{i\phi} \subset \partial \D :  \frac{ n(1-s^2)}{1+s^2-2s\cos(\phi - \psi)} \leq 1 \right \}.\]
\end{definition}

\begin{lemma}
\label{lem:j}
The set $K$ is empty for $s < \frac{n-1}{n+1}$, the single point $e^{i(\psi+\pi)}$ for $s=\frac{n-1}{n+1}$ and is an arc in $\partial \D$ centred at $e^{i(\psi+\pi)}$ for $s > \frac{n-1}{n+1}$. 
\end{lemma}

\begin{proof}
By \eqref{eq:deriv}, if $s$ is fixed, the smallest value that $|B'(z)|$ takes on $\partial \D$ is when $z=e^{i(\psi+\pi)}$ and is given by $\frac{n(1-s)}{1+s}$. This is decreasing in $s$ and is equal to $1$ when $s = \frac{n-1}{n+1}$. Hence the lemma follows from this observation and from the symmetry of the expression in \eqref{eq:deriv}.
\end{proof}

\begin{definition}
If $s>\frac{n-1}{n+1}$, so that $K$ is a non-empty arc, we denote the endpoints of $K$ by $e^{i\phi_1}$ and $e^{i\phi_2}$ with $\phi_1<\phi_2$ under the convention that $\arg(z) \in [\psi,\psi+2\pi)$.
\end{definition}
In particular, $B'(e^{i\phi_1}) = B'(e^{i\phi_2}) =1$.
Given an arc $I$ in $\partial \D$, we denote by $|I|$ the arc-length of $I$.

\begin{figure}[h]
\begin{center}
\includegraphics{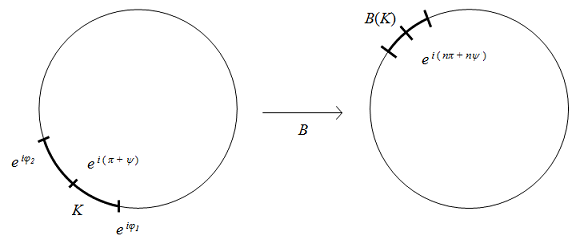}
\caption{Diagram showing the action of $B$ on $K$.}\label{fig:1}
\end{center}
\end{figure}

\begin{lemma}
\label{lem:fj}
With the notation above, if $s\geq \frac{n-1}{n+1}$, $|K| = 2\pi -2\cos^{-1}(t)$, where
\[ t=t(s) = \frac{1-n+(1+n)s^2}{2s},\]
and $|B(K)| = n(2\pi - 2\cos^{-1}(u))$, where
\[u = u(s) = \frac{1-n-(1+n)s^2}{2ns}.\]
\end{lemma}

See Figure \ref{fig:1} for $K$ and $B(K)$ when $s>\frac{n-1}{n+1}$.

\begin{proof}
By Lemma \ref{lem:j}, since $K$ is an arc and the endpoints of $K$ are those $e^{i\phi}$ for which
\[ \cos(\phi - \psi) = \frac{1-n+(1+n)s^2}{2s},\]
the first part follows. For the second part, since $B(z)=A(z)^n$, $B(K)$ will be an arc. To find $|B(K)|$, we just need to find $|A(K)|$ and then multiply by $n$. Note that since $|B'(z)|\leq 1$ for $z\in K$, $|B(K)| \leq |K| <2\pi$.
Now,
\begin{align*}
 e^{-i\psi}A(e^{i\phi}) &= \frac{ e^{-i\psi} ( e^{i\phi} - se^{i\psi})}{1-se^{-i\psi}e^{i\phi} }\\
&= \frac{ (1+s^2)\cos(\phi-\psi) -2s +i(1-s^2)\sin(\phi - \psi) }{1+s^2-2s\cos(\phi-\psi)}
\end{align*}
Hence the endpoints of $K$ are mapped by $e^{-i\psi}A(z)$ onto
\[\frac{(1+s^2)t -s \pm  i(1-s^2)\sqrt{1-t^2} }{1+s^2-2st}, \]
which has real part
\[ \frac{ (1+s^2)\left ( \frac{1-n+(1+n)s^2}{2s}\right ) -2s }{1+s^2-2s\left ( \frac{1-n+(1+n)s^2}{2s} \right )  } = \frac{1-n-(1+n)s^2}{2ns} =:u. \]
Since $z\mapsto e^{-i\psi }z$ is just a rotation and does not change arc-length, it follows that $|A(K)| = 2\pi - 2\cos^{-1}(u)$. Since $z\mapsto z^n$ multiplies arc-length by $n$, the second part of the lemma follows.
\end{proof}

\begin{lemma}
\label{lem:j'}
For $s\in (\frac{n-1}{n+1},1)$, let $p(s) = |K|$ and $q(s) = |B(K)|$. Then $p'(s) > q'(s)$ and so $|K|$ grows faster than $|B(K)|$.
\end{lemma}

It follows from Lemma \ref{lem:fj} that as $s\to 1$, $|K| \to 2\pi$ and $|B(K)| \to 0$, but this lemma tells us more.

\begin{proof}
By Lemma \ref{lem:fj}, we have
\[ p'(s) = \frac{2t'(s)}{\sqrt{1-t^2}} ,\quad q'(s) = \frac{2nu'(s)}{\sqrt{1-u^2}}.\]
First, we have
\[ t'(s) = \frac{n-1+(1+n)s^2}{2s^2}\]
and
\begin{align*} 
1-t^2 &= 1 - \left ( \frac{1-n+(1+n)s^2 }{2s } \right )^2 \\
&= \frac{4s^2 - (1-n)^2 -2(1-n^2)s^2 - (1+n^2)s^4}{4s^2} \\
&= \frac{ (-(1-n)^2 + (1+n)^2s^2 )(1-s^2)}{4s^2}.\end{align*}
Therefore after simplifying we have
\begin{equation} \label{eq:j'1} p'(s) = \frac{ 2(n-1 + (n+1)s^2 )}{s\sqrt{1-s^2} \sqrt{ -(1-n)^2 + (1+n)^2s^2} }.\end{equation}
Next, we have
\[ u'(s) = \frac{ n-1-(1+n)s^2}{2ns^2}\]
and
\begin{align*}
1-u^2 &= 1-  \left ( \frac{1-n+(1+n)s^2 }{2ns } \right )^2 \\
&= \frac{4n^2s^2 - (1-n)^2 -2(1-n^2)s^2 - (1+n^2)s^4}{4n^2s^2} \\
&= \frac{ (-(1-n)^2 + (1+n)^2s^2 )(1-s^2)}{4n^2s^2}.
\end{align*}
We therefore have
\begin{equation} \label{eq:j'2} q'(s) = \frac{2n(n-1-(n+1)s^2)}{s\sqrt {1-s^2}\sqrt{ -(1-n)^2 + (1+n)^2s^2}  }.\end{equation}
Comparing \eqref{eq:j'1} and \eqref{eq:j'2}, and denoting $C$ by a positive function of $n,s$ we have
\begin{align*}
p'(s)-q'(s) &= C( n-1 + (n+1)s^2 - n(n-1)+n(n+1)s^2 ) \\
&= C( -(n-1)^2 + (n+1)^2s^2 ) >0
\end{align*}
for $s>\frac{n-1}{n+1}$.
Since we are restricting to $s>\frac{n-1}{n+1}$, this proves the lemma.
\end{proof}

\subsection{Classifying Blaschke products}

With these lemmas in hand, we can move on to discussing a classification of Blaschke products of the form $B(z) = A(z)^n$.

If $s=0$ then $0$ is a superattracting fixed point of $B$ and $B$ is elliptic. We want to show that there exists some $s_0$ so that for $0\leq s<s_0$, $B$ is elliptic, for $s=s_0$, $B$ is parabolic and for $s_0<s<1$, $B$ is hyperbolic. We note that there will be cases where $s_0=1$ and so we always have ellipticity.

Recall the M\"obius transformation $A(z) = \frac{z-se^{i\psi}}{1-se^{-i\psi}z}$. The dynamics of $A$ are easily understood: $e^{i\psi}$ is a repelling fixed point and $e^{i(\psi + \pi)}$ is an attracting fixed point. Since $B(z)=A(z)^n$, this means that the arc $K$, which we recall has centre $e^{i(\psi + \pi)}$, is contracted by $A$ and expanded for $z\mapsto z^n$. Hence to find when $B$ is parabolic, we need to determine when one of the endpoints of $K$ are fixed by this process.

We first give the case where the endpoints are never fixed.

\begin{lemma}\label{lem:1}
Suppose that $(n-1)\psi  \in \{ 2k \pi :k\in\Z \}$ if $n$ is even or $(n-1)\psi  \in \{ (2k+1)\pi : k\in \Z \}$ if $n$ is odd.
Then $B$ is always elliptic.
\end{lemma}

\begin{figure}[h]
\begin{center}
\includegraphics{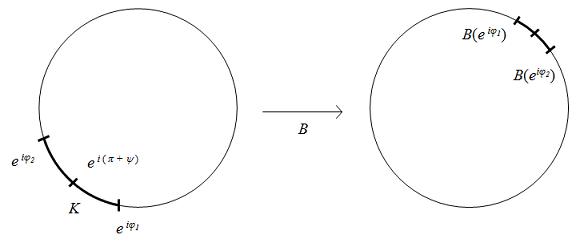}
\caption{Diagram showing the situation for Lemma \ref{lem:1}.}\label{fig:3}
\end{center}
\end{figure}

\begin{proof}
The point $e^{i(\psi + \pi)}$ is mapped to $e^{i[n(\psi + \pi) ]}$ under $B$. The conditions of the lemma state that this image is exactly the repelling fixed point of $A$ and is exactly opposite $e^{i(\psi + \pi)}$. Now, since $K$ is symmetric about $e^{i(\psi + \pi)}$, the arc-length of $K$ converges to $2\pi$ as $s\to 1$ and $K$ is contracted by $B$, no point of $K$ can be fixed by $B$. See Figure \ref{fig:3}. Therefore $B$ is never parabolic or hyperbolic and the conclusion follows.
\end{proof}

\begin{lemma}\label{lem:2}
Suppose that $(n-1)\psi  \notin \{ 2k \pi :k\in\Z \}$ if $n$ is even or $(n-1)\psi  \notin \{ (2k+1)\pi : k\in \Z \}$ if $n$ is odd. Then there exists $s_0 \in [\frac{n-1}{n+1},1)$ such that:
\begin{itemize}
\item for $0\leq s <s_0$, $B$ is elliptic,
\item for $s=s_0$, $B$ is parabolic,
\item for $s_0<s<1$, $B$ is hyperbolic.
\end{itemize}
\end{lemma}

\begin{proof}
The idea is as follows. Complementary to Lemma \ref{lem:1}, this time the image of $e^{i(\psi + \pi)}$ is not exactly opposite $e^{i(\psi + \pi)}$. First, if $B$ fixes $e^{i(\psi + \pi)}$, then by Lemma \ref{lem:j}, for $s = \frac{n-1}{n+1}$, $B$ is parabolic; for $0\leq s < \frac{n-1}{n+1}$, $B$ is elliptic; and for $s>\frac{n-1}{n+1}$, $B$ is hyperbolic. 

Next, $B(e^{i(\psi + \pi)})$ is either in the arc with argument $(\psi + \pi,\psi + 2\pi)$ or $(\psi, \psi + \pi)$. Without loss of generality we may assume it is the former, that is, $\phi_0 \in (\psi+\pi, \psi + 2\pi)$ where $B(e^{i(\psi + \pi)}) = e^{i\phi_0}$. See Figure \ref{fig:4} for reference to the arguments we give below. For values of $s$ just larger than $\frac{n-1}{n+1}$, $K$ and $B(K)$ are disjoint arcs centred at $e^{i(\psi + \pi)}$ and $e^{i\phi_0}$ respectively. 

\begin{figure}[h]
\begin{center}
\includegraphics[width=6.5in]{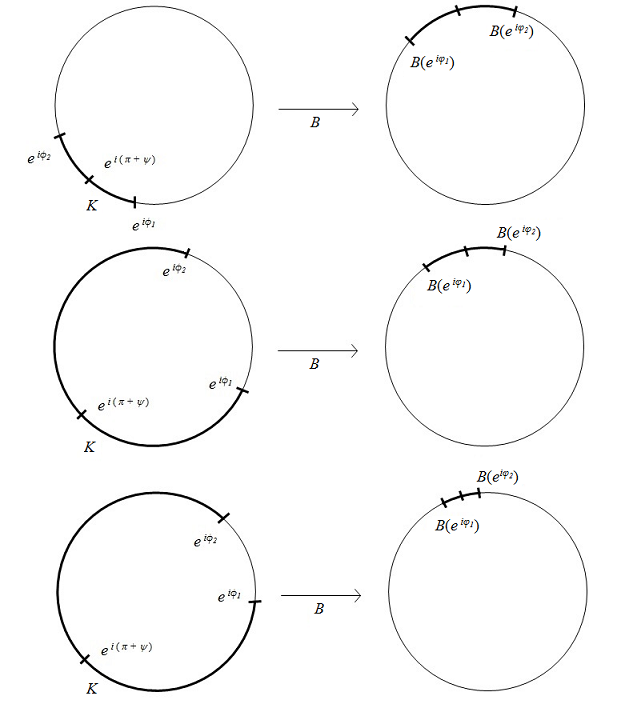}
\caption{Diagram showing the three situations that arise in Lemma \ref{lem:2}, the elliptic case ($s<s_0$) with there are no fixed points on $B$ in $K$ (top), the parabolic case ($s=s_0$) where $B(e^{i\phi_2}) = e^{i\phi_2}$ (middle) and the hyperbolic case ($s>s_0$) where $B(K)$ is strictly contained in $K$ (bottom).}\label{fig:4}
\end{center}
\end{figure}

Recall $\phi_1,\phi_2$ are the endpoints of $K$, and that we are assuming that $\phi_2$ and $\phi_0$ are both in the same semicircle with argument in $(\psi + \pi,\psi + 2\pi)$.
As $s$ increases, $|K|$ increases and we have that both $|K|\to 2\pi$ and $|B(K)| \to 0$ as $s\to 1$. By Lemma \ref{lem:j'}, $|K|$ increases faster than $|B(K)|$ as $s$ increases. Hence there exists one, and only one, $s_0 \in ( \frac{n-1}{n+1} ,1)$ such that $B(e^{i\phi_2}) = e^{i\phi_2}$.
By construction, when this occurs, $e^{i\phi_2}$ is a fixed point of $B$ for which $B'(e^{i\phi_2}) =1$ and hence $B$ is parabolic.

For $s<s_0$, if $B$ did have a fixed point in $K$, then it could not be an endpoint by construction. Suppose this fixed point is $e^{i\xi}$, and it must lie between $e^{i(\psi + \pi)}$ and $e^{i\phi_2}$. Then we have that the arc length of the interval in $\partial \D$ from $e^{i\xi}$ to $e^{i\phi_2}$ is strictly less than the arc length of the interval from $B(e^{i\xi}) = e^{i\xi}$ to $B(e^{i\phi_2})$. This contradicts the fact that $|B'(z)|\leq 1$ for $z\in K$.

For $s>s_0$, by Lemma \ref{lem:j'} the image of the arc $(\psi + \pi, \phi_2)$ strictly contains its image and hence there exists a fixed point with $|B'(z)| <1$. Hence $B$ is hyperbolic. 
\end{proof}

It now follows from Lemmas \ref{lem:1} and \ref{lem:2} that $\{ w\in \D:B_w \in \widetilde{\mathcal{E}_n} \}$ is an open subset of $\D$ which is starlike about $0$. It is clear that the relative boundary of this set in $\D$ consists of parabolic elements and so it remains to investigate these parabolic elements.

\begin{proof}[Proof of Theorem \ref{cor:1}]
By Theorem \ref{thm:f''}, parameters in the relative boundary of $\mathcal{E}_n$ in $S_n$ give rise to parabolic Blaschke products with $J(B) = \partial \D$ only when $B''(z_0) =0$, where $z_0$ is the Denjoy-Wolff point of $B$.
Differentiating \eqref{eq:f'}, we have
\begin{align*} B''(z) &= nA(z)^{n-2} \left ( (n-1)[A'(z)]^2 + A(z)A''(z) \right ) \\
&= nA(z)^{n-2} \left [ \frac{(n-1) (1-|w|^2)^2 }{(1-\overline{wz})^4} + \left ( \frac{z-w}{1-\overline{w}z} \right ) \left (\frac{ 2\overline{w}(1-|w|^2)}{(1-\overline{w}z)^3} \right )  \right ] \\
&= \frac{n(1-|w|^2) A(z)^{n-2}}{(1-\overline{w}z)^4} \left ( (n-1)(1-|w|^2) + 2\overline{w}(z-w) \right).
\end{align*}
For $z\in \partial \D$, $A(z) \neq 0$ and so we have to analyze the factor $  (n-1)(1-|w|^2) + 2\overline{w}(z-w)$. This is zero when
\[ z = \frac{ 2|w|^2 + (1-n) +(n-1)|w|^2 }{2\overline{w}} = \frac{(1-n)+(n+1)|w|^2}{2\overline{w}}.\]
Since $|z|=1$, this leads to
\[ (n+1)|w|^2-2|w|+(1-n) =0,\]
which has solutions $|w|=1$ and $|w| = \frac{n-1}{n+1}$. This first of these is inadmissible, but the second is allowable.
By Lemmas \ref{lem:j} and \ref{lem:2}, the only situation when a parabolic parameter $w$ has $|w|=\frac{n-1}{n+1}$ is when $e^{i(\psi + \pi)}$ is fixed by $B$. Since $e^{i(\psi + \pi)}$ is fixed by $A$, this leads to the condition that
$(n-1)\psi  =0 \: (\operatorname{mod} 2\pi )$ if $n$ is odd or $(n-1)\psi  = \pi \: ( \operatorname{mod} 2 \pi)$ if $n$ is even.

Hence there are $n-1$ distinct points on the relative boundary of $\widetilde{\mathcal{E}_n}$, corresponding to where this boundary intersects $\{ w : |w| = \frac{n-1}{n+1} \}$, for which the corresponding Blaschke product is parabolic with $J(B) = \partial \D$. For any other parabolic parameter, Theorem \ref{thm:f''} implies that $J(B)$ is a Cantor subset of $\partial \D$.

Since these $n-1$ points are obtained through rotations by $(n-1)$'th roots of unity, it follows there is just one point of the relative boundary of $\mathcal{E}_n$ that is in $\mathcal{M}_n$.
\end{proof}

We end with the observation that when $n=2$, there is one point in $\partial \mathcal{E}_2\cap\mathcal{M}_2$, given by $|w| = 1/3$. This tells us that
\[ B(z) =  \left ( \frac{ z + 1/3}{1+ z/3} \right )^2\]
is parabolic with Denjoy-Wolff point $1$ and has $J(B) = \partial \D$.

\end{document}